\documentclass[11pt,a4paper]{article}

\usepackage{amsthm,amsmath,amsfonts,amssymb}
\usepackage[english]{babel}
\usepackage{graphicx}
\usepackage{verbatim}
\usepackage{bm}
\usepackage{color}

\newcommand\blfootnote[1]{%
  \begingroup
  \renewcommand\thefootnote{}\footnote{#1}%
  \addtocounter{footnote}{-1}%
  \endgroup
}

\newtheorem{proposition}{Proposition}[section]
\newtheorem{theorem}{Theorem}[section]
\newtheorem{lemma}{Lemma}[section]

\newtheorem{corollary}{Corollary}[section]

\newtheorem{definition}{Definition}[section]

\newtheorem{remark}{Remark}[section]

\newtheorem{algorithm}{Algorithm}[section]

\def\dist{{\rm dist\,}}
\def\diam{{\rm diam\,}}

\def\vec0{\mbox{\boldmath $0$}}

\title{Identifying codes in line digraphs
\thanks{This research is partially supported by AGAUR under project 2017SGR1087. }
}

\author{ C. Balbuena$^{1}$, C. Dalf\'o$^{2}$, B. Mart\'\i nez-Barona$^{1}$
 \\[2ex]
$^1${\footnotesize Departament d'Enginyeria Civil i Ambiental}\\
{\footnotesize Universitat Polit\`ecnica de Catalunya}\\
$^2${\footnotesize Departament de Matem\`atica}\\
{\footnotesize Universitat de Lleida }\\
{\footnotesize e-mails:\{m.camino.balbuena,berenice.martinez\}@upc.edu, cristina.dalfo@matematica.udl.cat}
}

\date{}

\begin{document}

\maketitle

\blfootnote{
\begin{minipage}[l]{0.3\textwidth} \includegraphics[trim=10cm 6cm 10cm 5cm,clip,scale=0.15]{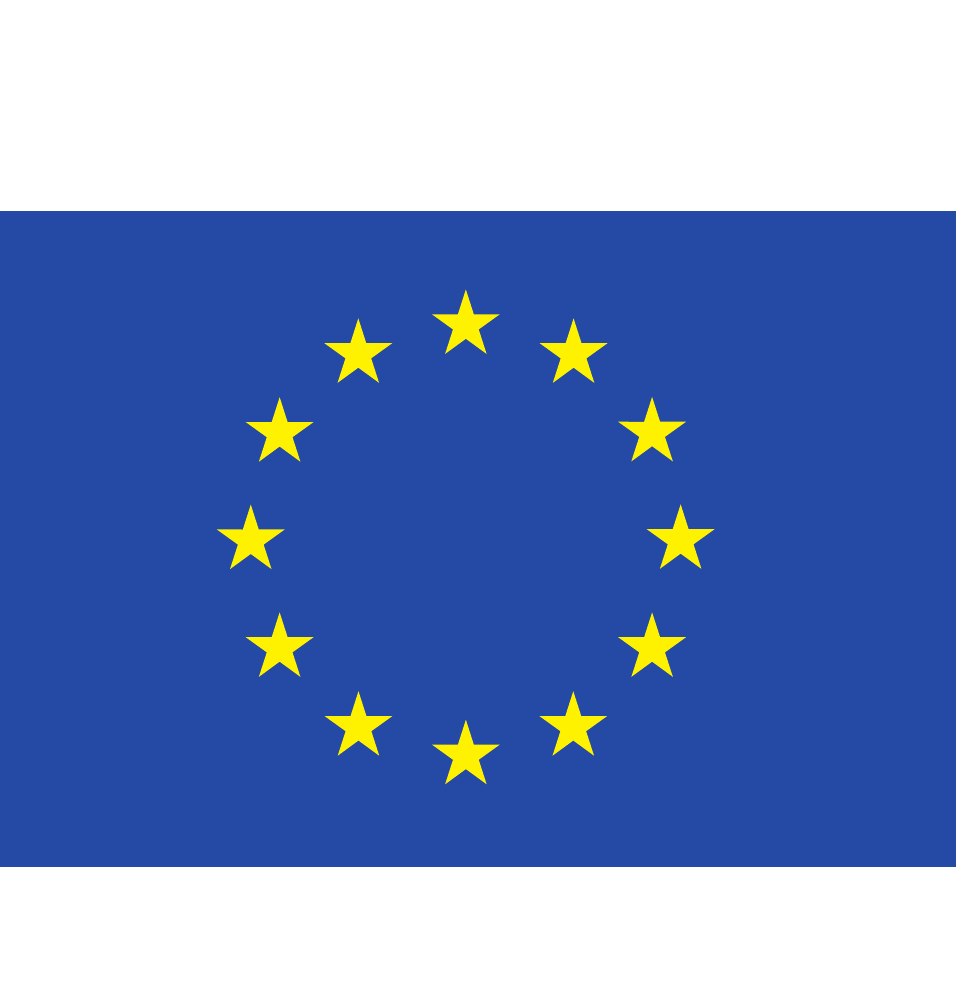} \end{minipage}  \hspace{-2cm} \begin{minipage}[l][1cm]{0.79\textwidth}
The research of the second author has also received funding from the European Union's Horizon 2020 research and innovation programme under the Marie Sk\l{}odowska-Curie grant agreement No 734922.
  \end{minipage}}

\begin{abstract}
Given an integer $\ell\ge 1$, a $(1,\le \ell)$-identifying code in a digraph is a dominating subset $C$ of
vertices such that all distinct subsets of vertices of cardinality at most $\ell$ have distinct closed in-neighbourhood within $C$. In this paper, we prove that every $k$-iterated line digraph of minimum in-degree at least 2 and $k\geq2$, or minimum in-degree at least 3 and $k\geq1$, admits a $(1,\le \ell)$-identifying code with $\ell\leq2$, and in any case it does not admit a $(1,\le \ell)$-identifying code for $\ell\geq3$. Moreover, we find that the identifying number of a line digraph is lower bounded by the size of the original digraph minus its order. Furthermore, this lower bound is attained for oriented graphs of minimum in-degree at least 2.
\end{abstract}



\section{Introduction}

In this paper we study the concept of $(1,\leq\!\ell)$-identifying codes in line digraphs, where $\ell\ge1$ is an integer.
In~\cite{BaDaMa17}, the authors studied the $(1,\leq\!\ell)$-identifying codes in digraphs, and gave some sufficient conditions for a digraph of minimum in-degree  $\delta^-\ge 1$ to admit a $(1,\le \ell)$-identifying code  for $\ell=\delta^-, \delta^-+1$.
Regarding line graphs, Foucaud, Gravier, Naserasr, Parreau, and Valicov~\cite{FGNPV13} studied $(1,\le 1)$-identifying codes  and Junnila and Laihonen~\cite{JL14} studied $(1,\le\ell)$-identifying codes for $\ell\ge 2$. 

We consider simple digraphs without loops or multiple edges. Unless otherwise stated, we follow the book by Bang-Jensen and Gutin \cite{BG07} for terminology and definitions.

Let $D$ be a digraph with vertex set $V(D)$ and arc set $A(D)$.
A vertex $u$ is \emph{adjacent to} a vertex $v$ if $(u,v)\in A(D)$. If both arcs $(u,v),(v,u)\in A(D)$,  then we say that they form a \emph{digon}. A digraph is \emph{symmetric} if $(u,v)\in A(D)$ implies $(v,u)\in A(D)$, so it can be studied as a graph. A digon is often referred as a \emph{symmetric arc} of $D$. An \emph{oriented graph} is a digraph without digons.
The \emph{out-neighborhood} of a vertex $u$ is $N^+(u)=\{v\in V: (u,v)\in A(D)\}$ and the \emph{in-neighborhood} of  $u$ is $N^-(u)=\{v\in V(D): (v,u)\in A(D)\}$. The {\em closed in-neighbourhood} of $u $ is $N^-[u]=\{u\}\cup N^-(u)$. Given a vertex subset $U\subset V(D)$, let $N^-[U]=\bigcup_{u\in U}N^-[u]$ and $N^+[U]=\bigcup_{u\in U}N^+[u]$.
\emph{A dominating set} is a subset of vertices $S\subseteq V$ such that $N^+[S]=V$.
The \emph{out-degree} of $u$ is $d^{+}(u)=|N^+(u)|$, and its \emph{in-degree}  $d^{-}(u)=|N^-(u)|$. We denote by $\delta^+=\delta^+(D)$ the minimum out-degree of the vertices in $D$, and by $\delta^-=\delta^-(D)$ the minimum in-degree. The minimum degree is $\delta=\delta(D)=\min\{\delta^+ , \delta^- \}$.
A digraph $D$ is said to be $d$-\emph{in-regular} if $|N^-(v)|=d$ for all $v\in V(D)$,  and  $d$\emph{-regular} if $|N^+(v)|=|N^-(v)|=d$ for all $v\in V(D)$. A path from $u $ to $v$ is denoted by $(u\ldots v)$ or also by $u\rightarrow v$. A path of order $k$ is called a $k$-path.
A digraph $D$ is said to be \emph{strongly connected} when, for any pair of vertices $u,v\in V(D)$, there always exists a $u\rightarrow v$ path.
For any pair of vertices $u,v\in V(D)$ we denote by $\dist(u,v)$ the distance from $u$ to $v$ in $D$, that is, $\dist(u,v)=\min\{k\mid\; \mbox{there is a $k$-path in $D$ from $u$ to $v$}\}$.
For each vertex $v\in V(D)$, we denote by $\omega^-(v)=\{(u,v)\in A(D)\}$ and $\omega^+(v)=\{(v,u)\in A(D)\}$.
For a natural number $k$, a $k$-cycle is a directed cycle of order $k$.

For a given integer $\ell\ge 1$, a vertex subset $C\subset V(D)$ is a {\em $(1,\leq\ell)$-identifying code} in $D$ if it is a dominating set and for all distinct subsets $X,Y\subset V(D)$, with $1\le |X|,|Y|\leq\ell$, we have
\begin{equation}\label{code}
N^-[X]\cap C \neq N^-[Y]\cap C.
\end{equation}
The definition of a $(1,\leq\ell)$-identifying code for graphs was introduced by Karpovsky, Chakrabarty and Levitin~\cite{KCL98}, and its definition can be obtained from (\ref{code}) by omitting the superscript signs minus. Thus, the definition for digraphs  is a natural extension  of the concept of  $(1,\le \ell)$-identifying codes in graphs. A $(1,\le 1)$-identifying code is known as an \emph{identifying code}.
Thus, an identifying code of a graph is a dominating set, such that any two vertices of the graph have distinct closed neighborhoods within this set.
Identifying codes model fault-diagnosis in multiprocessor systems, and these are used in other applications, such as the design of emergency sensor networks. For more information on these applications, see Karpovsky, Chakrabarty, and Levitin~\cite{KCL98} and Laifenfeld, Trachtenberg, Cohen and Starobinski~\cite{LTCS07}.

Note that if $C$ is a $(1,\leq \ell)$-identifying code in a digraph $D$, then the whole set of vertices $V(D)$ also is. Thus, a digraph $D$ admits some $(1,\leq\ell)$\emph{-identifying code} if and only if
for all distinct subsets $X,Y\subset V(D)$ with $|X|,|Y|\leq \ell$, we have
\begin{equation}
N^-[X] \neq N^-[Y].
\end{equation}


We recall that a \emph{transitive tournament} of 3 vertices is denoted by $TT_3$, as shown in Figure~\ref{fig:tt3}.
\begin{remark}
\label{TT3}
Let $D$ be a $TT_3$-free digraph. Then, for every arc $(x,y)$ of $D$, we have $N^-(x)\cap N^-(y)=\emptyset$ and $N^+(x)\cap N^+(y)=\emptyset$.
\end{remark}

\begin{figure}[t]
    \begin{center}
  \includegraphics[width=2cm]{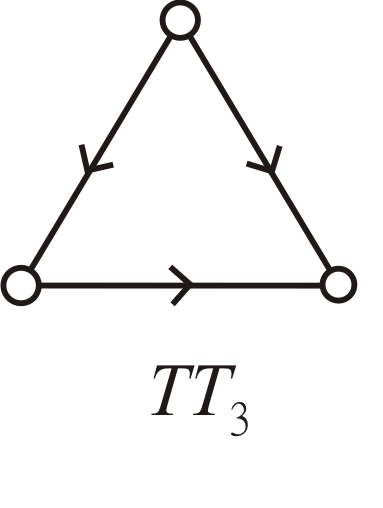}
  \end{center}
  \vskip-1cm
  \caption{A transitive tournament on 3 vertices.}
  \label{fig:tt3}
\end{figure}

\section{Identifying codes in line digraphs}

In the line digraph $LD$ of a digraph $D$, each vertex  represents an arc of $D$. Thus, $V(LD)=\{ uv : (u,v) \in A(D)\} $; and  a vertex $uv$ is adjacent to a vertex $wz$ if and only if $v=w$, that is, when  the arc $(u,v)$ is adjacent to the arc $(w,z)$ in $D$. For any integer $k\geq1$,  the {\it  $k$-iterated line digraph} $L^kD$ is defined recursively by  $L^kD=LL^{k-1}D$, where $L^0D=D$. From the definition, it is evident that the order of $LD$  equals the size of $D$, that is, $|V(LD)| = |A(D)|$.
Due to the bijection between the set of arcs in the digraph D and the set of vertices in the digraph $LD$, when it is clear from the context, we use $uv$ to denote both the arc in $A(D)$ and the vertex in $V(LD)$. Hence, for each vertex $v\in V(D)$, the set of arcs $\omega^+(v)$ in $D$ corresponds to a set of vertices in $LD$.
If $D$ is a strongly connected digraph different from a directed cycle with minimum degree $\delta$,  then the iterated line digraph $L^k D$ has minimum degree $\delta$ and diameter $\diam(L^kD)=\diam(D)+k$.
See  Aigner \cite{A67}, Fiol, Yebra, and Alegre \cite{FYA84}, and Reddy,
Kuhl, Hosseini, and Lee \cite{RKHL82}.

A large known family of digraphs obtained with the line digraph technique is the family of Kautz digraphs. The {\it Kautz digraph}  of degree $d$ and diameter $k$ is defined as the $(k-1)$-iterated line digraph of the symmetric complete digraph of $d+1$ vertices $K_{d+1}$, that is, $K(d,k)\cong L^{k-1}K_{d+1}$.
For instance, the Kautz digraph $K(2,2)$, shown in Figure \ref{fig:exemple}, is the line digraph of the symmetric complete digraph on three vertices.

Line digraphs were characterized by Heuchenne \cite{H64} with the following property: A digraph $D$ is a line digraph if and only if it has no multiple arcs, and for any pair of vertices $u$ and $v$, either $N^-(u)\cap N^-(v)=\emptyset$ or $N^-(u)= N^-(v)$. A similar characterization is obtained replacing $N^-$ by $N^+$.

The \emph{semigirth} $\gamma$ was defined by F\`{a}brega and Fiol~\cite{FF89}  as follows.
\begin{definition}\label{semigirth} ~\cite{FF89}
Let $D$ be a digraph with minimum degree $\delta$. Let $\gamma =\gamma(D)$, for $1\leq\gamma\leq \diam (D)$, be the greatest integer such that, for any $x,y\in V(G)$:
\begin{enumerate}
\item if $\dist(x,y)< \gamma $, the shortest $x\rightarrow y$ path is unique and there are no paths of length $\dist(x,y)+1$; \item if $\dist(x,y)=\gamma $, there is only one shortest $x\rightarrow y$ path.
\end{enumerate}
\end{definition}
Note that, as $D$ has no loops, $\gamma\ge 1$. In \cite{FF89} it was also proved that, if $D$ is a strongly connected digraph without loops and different from a directed cycle, then $\gamma(L^kD)=\gamma+k$.

\begin{figure}[t]
\vskip-.75cm
    \begin{center}
  \includegraphics[width=10cm]{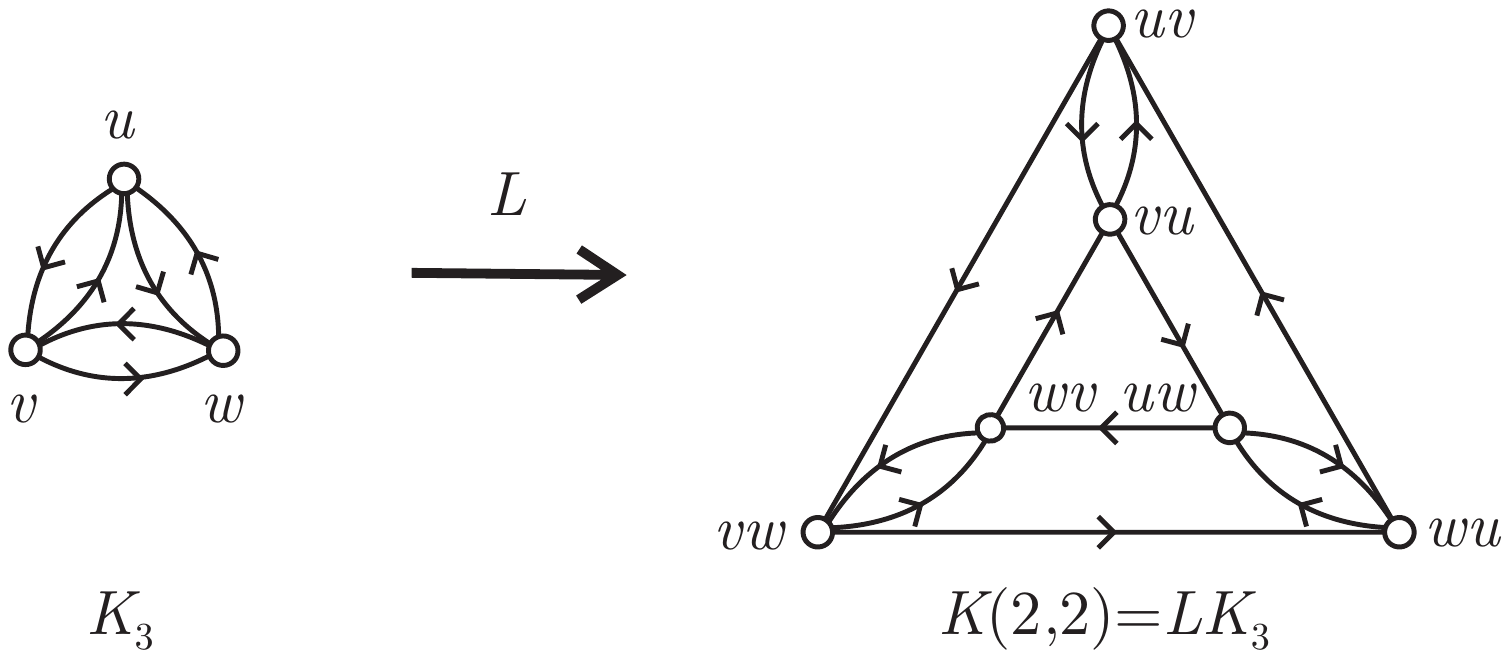}
  \end{center}
  \vskip-11.5cm
  \caption{The Kautz digraph $K(2,2)$ as the line digraph of the symmetric complete digraph $K_{3}$.}
  \label{fig:exemple}
\end{figure}

From now on we are going to consider strongly connected digraphs.

\begin{remark}
 \label{RemarkObsLineDig}
If $D$ is a digraph with $\gamma(D)\ge 2$, then:\begin{enumerate}
\item[(i)] $D$ is $TT_3$-free,
\item[(ii)] the paths of length two are unique.
\end{enumerate}
\end{remark}
Observe that for any line digraph $LD$ different from a directed cycle, $\gamma(LD)\ge 2$, therefore by Remark \ref{RemarkObsLineDig}, any line digraph is $TT_3$-free. As a consequence, we can write the following result.

\begin{proposition}\label{uno}
The line digraph of a strongly connected digraph of order at least $3$ admits a $(1,\leq1)$-identifying code. \quad $\Box$
\end{proposition}


The following result is a direct consequence of Remark \ref{RemarkObsLineDig} $(ii)$ and the definition of line digraph.

\begin{lemma}
\label{LemmLDPropieties}
\label{CorObsLineDig}
Let $LD$ be a line digraph.
\begin{enumerate}
\item[(i)] If $u,v\in V(D)$  are two different vertices such that $N^-(u)\cap N^-(v)\neq\emptyset$, then $N^+(u)\cap N^+(v)=\emptyset$.
\item[(ii)] There are no two digons incident with the same vertex.
\end{enumerate}
\end{lemma}

\begin{figure}[t]
    \begin{center}
  \includegraphics[width=10cm]{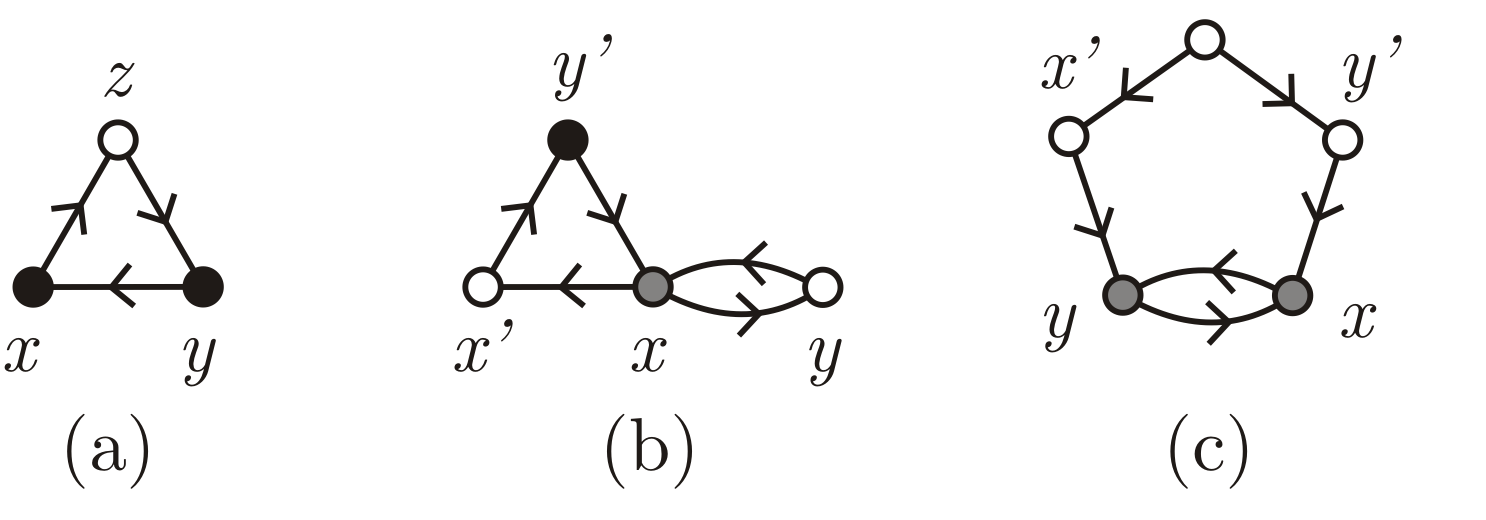}
  \end{center}
  \vskip-.7cm
  \caption{ The forbidden subdigraphs of Theorem \ref{LineIdentifyingdelta1} and Corollary \ref{LD(12)IDcode}.}
  \label{fig:conf-prohibidesLD(abc)}
\end{figure}

In~\cite{BaDaMa17}, the authors proved that if $D$ is a digraph of minimum in-degree $\delta^-$ admitting a $(1,\leq\delta^-+1)$-identifying code, then the vertices of minimum in-degree does not lay on a digon. In the following theorem, we give sufficient and necessary conditions for a line digraph to admit a $(1,\leq2)$-identifying code.

\begin{theorem}
\label{LineIdentifyingdelta1}
Let $LD$ be a line digraph different from a 4-cycle and such that the vertices of in-degree 1 (if any) does not lay on a digon.
 Then, $LD$ admits a $(1,\leq2)$-identifying code if and only if
 $LD$ satisfies the following conditions:
 \begin{itemize}
\item[$(i)$] there are no 3-cycles with at least 2 vertices of in-degree 1 (see Figure \ref{fig:conf-prohibidesLD(abc)} (a) where the vertices of in-degree one are indicated in black color);
\item[$(ii)$] there do not exist four vertices $x,x', y$ and $y'$ such that $N^-(x)=\{y,y'\}$, $N^-(y')=\{x'\}$ and $x\in N^-(x')\cap N^-(y)$ (see Figure \ref{fig:conf-prohibidesLD(abc)} (b) where the vertices of in-degree one are indicated in black color and the vertices of in-degree two in gray color);
\item[$(iii)$] there do not exist two vertices $x,y\in V(LD)$ such that $N^-(x)=\{y,y'\}$, $N^-(y)=\{x,x'\}$ and $N^-(x')\cap N^-(y')\neq\emptyset $ (see Figure \ref{fig:conf-prohibidesLD(abc)} (c) where the vertices of in-degree two are indicated in gray color).
\end{itemize}
\end{theorem}

\begin{proof}
Let $LD$ be a line digraph satisfying the hypothesis of the theorem. First, suppose that $LD$ does not satisfy $(i)$. Hence, let $(z,y,x,z)$ be a 3-cycle such that $d^-(x)=1=d^-(y)$ (see Figure \ref{fig:conf-prohibidesLD(abc)} (a)). Then, $N^-[\{x,z\}]=\{x,y\}\cap N^-[z]=N^-[\{y,z\}]$, implying that $LD$ does not admit an identifying code.
Second, suppose that $LD$ does not satisfy $(ii)$. Let $X=\{x,x'\}$ and $Y=\{y,y'\}$, where $x,x', y, y'$ are four different vertices of $LD$ such that $N^-(x)=\{y,y'\}$, $N^-(y')=\{x'\}$, and $x\in N^-(x')\cap N^-(y)$ (see Figure \ref{fig:conf-prohibidesLD(abc)} (b)). Hence, by the Heuchenne's condition $N^-(x')=N^-(y)$, it follows that
\begin{equation*}
\begin{split}
  N^-[X]&=N^-(x)\cup N^-(x')\cup\{x,x'\}\\
  &=\{y,y'\}\cup N^-(y)\cup \{x,x'\}\\
  &=\{y,y'\}\cup N^-(y)\cup \{x'\}\\
  &=\{y,y'\}\cup N^-(y)\cup N^-(y')\\
  &=N^-[Y].
  \end{split}
\end{equation*}
Therefore, $LD$ does not admit a $(1,\leq2)$-identifying code.
Now, suppose that $LD$ does not satisfy $(iii)$. Let $X=\{x,x'\}$ and $Y=\{y,y'\}$, where $N^-(x)=\{y,y'\}$, $N^-(y)=\{x,x'\}$, and $N^-(x')\cap N^-(y')\neq\emptyset$ (see Figure \ref{fig:conf-prohibidesLD(abc)} (c)).  Since, by the Heuchenne's condition $N^-(x')=N^-(y')$, it follows that
\begin{equation*}
\begin{split}
  N^-[X]&=N^-(x)\cup N^-(x')\cup\{x,x'\}\\
  &=\{y,y'\}\cup N^-(y')\cup N^-(y)\\
  &=N^-[Y].
  \end{split}
\end{equation*}
Therefore, $LD$ does not admit a $(1,\leq2)$-identifying code.

For the converse, let $X,Y\subset V(LD)$ be two different subsets such that $1\leq|X|\leq|Y|\leq2$ and $N^-[X]=N^-[Y]$.
By Proposition~\ref{uno}, $|Y|=2$.
If $|X|=1$, say $X=\{x\}$, then for all $y\in Y\setminus X$, since $N^-[Y]=N^-[X]=N^-[x]$, it follows $N^-[y]\subset N^-(x)$. Hence, $y\in N^-(x)$. If $d^-(y)=1$, there is at least one vertex $z\in N^-(y)\cap N^-(x)$ because there are no vertices of in-degree 1 laying on a digon,
and clearly the same happens if $d^-(y)\ge 2$,
  reaching a contradiction to Remark \ref{RemarkObsLineDig} $(i)$.
Hence, $|X|=2$, and there are two cases to be considered.
First, let us suppose that $X\cap Y\neq\emptyset$.  Let $X=\{x,z\}$ and $Y=\{y,z\}$.
If there is an arc between $x$ and $y$, say $yx\in A(LD)$, then by Remark \ref{RemarkObsLineDig} $(i)$, $N^-(x)\cap N^-(y)=\emptyset$. Then, $N^-(y)\subseteq N^-[z]$ and $N^-(x)\subseteq N^-[z]\cup\{y\}$.
Consider $d^-(x)\geq2$ and let $u\in N^-(x)\setminus\{y\}$. Hence, $u\in N^-[z]$. If $u=z$, then $N^-(x)\cap N^-(z)=\emptyset$. Hence, by Remark \ref{RemarkObsLineDig} $(i)$ and $(ii)$, $N^-[z]\cap N^-(y)=\emptyset$, implying that $N^-(y)=\emptyset$, a contradiction since $\delta^-\geq1$. Then, $u\in N^-(z)\cap N^-(x)$ implying, by the Heuchenne's condition that $N^-(z)= N^-(x)$, hence $y\in N^-(z)$, which is a contradiction since $N^-(y)\subseteq N^-[z]$.
Now suppose that $d^-(x)=1$, then $N^-(x)=\{y\}$. Since $x\in N^-[Y]$ and $x$ does not lay on a digon, $x\in N^-(z)$. Since, $x\notin N^-(y)$, $N^-(y)\cap N^-(z)=\emptyset$, implying that $N^-(y)=\{z\}$ because $N^-(y)\subseteq N^-[z]$. Therefore, $(x,z,y,x)$ is a 3-cycle of $LD$ with two vertices of in-degree 1, implying that $LD$ does not satisfy $(i)$.
Now, suppose that there are no arcs between $x$ and $y$. Since $x\in N^-[Y]$ and $y\in N^-[X]$, it follows that $x,y\in N^-(z)$. Hence, by Remark \ref{RemarkObsLineDig} $(i)$ and $(ii)$, $N^-(x)\cap (N^-(z)\cup N^-(y))=\emptyset$, implying that $N^-(x)=\{z\}$, a contradiction since there are no vertices of in-degree 1 laying on a digon.

Now let $X\cap Y=\emptyset$, with $X=\{x,x'\}$ and $Y=\{y,y'\}$.
In order to $y\in N^-[X]$, assume that $y\in N^-(x)$, that is, $yx\in A(LD)$. Then, by Remark \ref{TT3}, $N^-(x)\cap N^-(y)=\emptyset$ implying that $N^-(y)\subseteq N^-(x')\cup\{x,x'\}$.
Since $x\in N^-[Y]$, there are two cases to be considered.

First, suppose that $x\in N^-(y)$. Then, $d^-(x),d^-(y)\geq2$, since both vertices lay on a digon.
If there is $u\in N^-(y)\setminus(X\cup Y)$, then $u\in N^-(x')$, implying that $x\in N^-(x')$ by the Heuchenne's condition. Hence, since $x'\in N^-[Y]$ and $N^-(x')=N^-(y)$, it follows that $x'\in N^-(y')$. Furthermore, $y'\in N^-(x')$ or $y'\in N^-(x)$. If $y'\in N^-(x')$, then by the Heuchenne's condition $N^-(x')\cap N^-(y')=\emptyset$. Moreover, by Remark \ref{RemarkObsLineDig}, $x,y\notin N^-(y')$, implying that $N^-[X]\cap N^-(y')=\{x'\}$. Hence, $d^-(y')=1$, because $N^-[X]=N^-[Y]$, a contradiction because $y'$ lay on a digon.
If $y'\in N^-(x)$, then $N^-(y')\cap (N^-(x)\cup N^-(x'))=\emptyset$, implying that $N^-(y')=\{x'\}$ and $N^-(x)=\{y,y'\}$. Therefore, $LD$ does not satisfy $(ii)$.
Reasoning similarly for $x$ as we did for $y$, we can assume that $N^-(y)\subseteq X\cup Y$ and $N^-(x)\subseteq X\cup Y$. If $x'\in N^-(x)$, then $x'\notin N^-(y)$, implying that $y'\in N^-(y)$. Then, by Lemma \ref{LemmLDPropieties} $(ii)$ and Remark \ref{RemarkObsLineDig} $(i)$, $x,y\notin N^-(x')\cup N^-(y')$, implying that there is a vertex $u\in (N^-(x')\cap N^-(y'))\setminus(X\cup Y)$. Therefore, $LD$ does not satisfy $(iii)$.
If $x'\in N^-(y)$, then $y'\notin N^-(y)$, implying that $y'\in N^-(x)$. Then, by Lemma \ref{LemmLDPropieties} $(ii)$ and Remark \ref{RemarkObsLineDig} $(i)$, $x,y\notin N^-(x')\cup N^-(y')$, implying that there is a vertex $u\in (N^-(x')\cap N^-(y'))\setminus(X\cup Y)$. Therefore, $LD$ does not satisfy $(iii)$.

Now, suppose that $x\in N^-(y')\setminus N^-(y)$. Then, $N^-(x)\cap(N^-(y)\cup N^-(y'))=\emptyset$, implying that $N^-(x)=\{y\}$. Hence, since $y'\in N^-[X]$, it follows that $y'\in N^-(x')$, implying that $N^-(y')\cap(N^-(x')\cup N^-(x))=\emptyset$, and consequently $N^-(y')\subseteq\{x,x'\}$. Since $x'\in N^-[Y]$, there are two cases to be considered.
If $x'\in N^-(y)$, then $N^-(x')\cap(N^-(y)\cup N^-(y'))$, and by Lemma \ref{LemmLDPropieties} $(ii)$ and since $LD$ is $TT_3$-free, $N^-(x)=\{y'\}$, implying that $d^-(y)=d^-(y')=1$. Hence, $LD$ would be a 4-cycle, since $LD$ is a strongly connected digraph, a contradiction.
Therefore, $x'\in N^-(y')$, implying that $d(x')\geq2$. Since $x\in N^-(y')\setminus N^-(y)$ and $y'\in N^-(x')\setminus N^-(x)$, it follows that $N^-(y)\cap N^-(y')=\emptyset$ and $N^-(x)\cap N^-(x')=\emptyset$, respectively. Then, $x'\notin N^-(y)$ and $y\notin N^-(x')$. Then, there is $u\in N^-(x')\setminus (X\cup Y)$, implying that $u\in N^-(y)$ and, hence, $LD$ does not satisfy $(iii)$. Thus, this completes the proof.\end{proof}


Notice that, according to the above theorem, if a line digraph  with minimum in-degree $\delta^-\geq2$ does not admit a $(1,\leq2)$-identifying code, then $\delta^-=2$. Note that $\gamma(L^kD)=k+1\ge 3$ if $k\ge 2$, which implies by Definition \ref{semigirth} that $L^kD$ does not contain two vertices satisfying the hypothesis of Theorem \ref{LineIdentifyingdelta1}. Therefore, we have the following result.

\begin{corollary}
\label{LD(12)IDcode}
  Let $L^kD$ be a line digraph with minimum in-degree $\delta^-\geq2$.
  \begin{itemize}
    \item[(i)] If $k\geq2$, then $L^kD$ admits a $(1,\leq2)$-identifying code.
    \item[(ii)] If $k=1$ and $\delta^-\geq3$, then $LD$ admits a $(1,\leq2)$-identifying code.
    \item[(iii)] If $D$ is a 2-in-regular digraph and $k\geq1$, then $L^kD$ admits a $(1,\leq2)$-identifying code if and only if $L^kD$ does not contain the subdigraph of Figure \ref{fig:conf-prohibidesLD(abc)} $(c)$.
  \end{itemize}
\end{corollary}

\begin{corollary}
\label{K2kIdentifying}
For each $n\geq 3$, the Kautz digraph $K(n,2)=LK_{n+1}$ admits a $(1,\leq 2)$-identifying code.
\end{corollary}
 By Corollary \ref{LD(12)IDcode} $(iii)$, the Kautz digraph $K(2,2)=LK_3$ (see Figure \ref{fig:exemple}) does not admit a $(1,\leq2)$-identifying code. Then, the condition $k\geq2$ in Corollary \ref{LD(12)IDcode} $(i)$ is necessary.

\begin{remark}\label{degree}Let $ D$ be a digraph with minimum in-degree $\delta^-\geq 2$. Then, there exists a vertex $u\in V( D)$ such that $d^+(u)\geq2$.
It is enough to observe that if  $d^+(u)<2$ for all $u\in V(D)$, then we would reach the contradiction:
$$
2|V(D)|\leq\sum_{v\in V(D)}d^-(v)=\sum_{v\in V(D)}d^+(v)\leq|V(D)|.
$$
Consequently, any line digraph $LD$ with minimum in-degree $\delta^-\geq 2$ contains at least two vertices with the same in-neighborhood by the Heuchenne's condition.
\end{remark}

\begin{proposition}
\label{(1,3)IDcodeLine}
Let $LD$ be a line digraph with minimum in-degree $\delta^-\geq2$, then $LD$ does not admit a $(1,\leq3)$-identifying code.
\end{proposition}

\begin{proof}
By Remark \ref{degree}, there are two different vertices $u,v\in V(LD)$ such that $N^-(u)=N^-(v)$. Moreover, $LD$ has $\delta^+\geq1$ because it is strongly connected. Let $w\in N^+(u)$, thus by Remark \ref{TT3}, $w\ne v$.  Then, $X=\{u,v,w\}$ and $Y=\{v,w\}$ are two different sets such that $N^-[X]=N^-[Y]$, implying that $LD$ does not admit a $(1,\leq3)$-identifying code.
\end{proof}

\section{The identifying number of a line digraph}

Foucaud, Naserasr, et al. \cite{FGNPV13}  characterized the digraphs that only admit as identifying code the whole set of vertices. Let us introduce the terminology used for this characterization.

Given two digraphs $D_1$ and $D_2$ on disjoint sets of vertices, we denote $D_1\oplus D_2$ the disjoint union of $D_1$ and $D_2$, that is, the digraph whose vertex set is $V(D_1)\cup V(D_2)$ and whose arc set is $A(D_1)\cup A(D_2)$.
Given a digraph $D$ and a vertex $x\notin V(D)$, $x\overrightarrow{\vartriangleleft}(D)$  is the digraph with vertex set $V(D)\cup\{x\}$, and whose arcs are the arcs of $D$ together with each arc $(x,v)$ for every $v\in V(D)$.
\begin{definition}
We define $(K_1,\oplus,\overrightarrow{\vartriangleleft})$ to be the closure of the one-vertex graph $K_1$ with respect to the operations $\oplus$ and $\overrightarrow{\vartriangleleft}$. That is, the class of all graphs that can be built from $K_1$ by repeated applications of $\oplus$ and $\overrightarrow{\vartriangleleft}$.
\end{definition}

Foucaud, Naserasr, et al. \cite{FNP13} proved that for any digraph $D$, $\overrightarrow{\gamma}^{ID}(D)=|V(D)|$ if and only if $D\in(K_1,\oplus,\overrightarrow{\vartriangleleft})$.
 Since, as they pointed out, every element $D\in(K_1,\oplus,\overrightarrow{\vartriangleleft})$ is the transitive closure of a rooted oriented forest, if $LD$ is a line digraph with minimum in-degree $\delta^-\geq2$, then $LD\notin (K_1,\oplus,\overrightarrow{\vartriangleleft})$. Hence, $\overrightarrow{\gamma}^{ID}(LD)\leq|V(LD)|-1$, where $\overrightarrow{\gamma}^{ID}(D)$ denotes the minimum size of an identifying code of a digraph $D$. Next, we establish better upper bounds on $\overrightarrow{\gamma}^{ID}(LD)$.

With this goal, we define the relation $\sim$ over the set of vertices $V(LD)$ as follows. For all $u,v\in V(LD)$, $u\sim v$ if and only if $N^-(u)=N^-(v)$.
 Clearly, $\sim$ is an equivalence relation. For any $u\in V(LD)$, let $[u]_{\sim}=\{v\in V(LD): v\sim u\}$.

\begin{lemma}
 \label{RepEnC}
Let $D$ be digraph and $C$ an identifying code of $LD$. Then, for any vertex $w\in V(LD)$, $$|[w]_\sim\setminus C|\leq1.$$
\end{lemma}
\begin{proof}
 Let $w\in V(LD)$ and $u,v\in[w]_\sim\setminus C$. Then, $N^-(u)=N^-(v)$ and, since $u,v\notin C$, it follows that $N^-[u]\cap C=N^-(u)\cap C=N^-(v)\cap C=N^-[v]\cap C,$ which is a contradiction if $u\neq v$.
\end{proof}

\begin{definition}
\label{defi-arc-idCode}
  Given a digraph $D$, a subset $C$ of $A(D)$ is an \emph{arc-identifying code} of $D$ if $C$ is both:
  \begin{itemize}
    \item  an arc-dominating set of $D$, that is, for each arc $uv\in A(D)$, $(\{uv\}\cup\omega^-[u])\cap C\neq\emptyset$, and
    \item  an arc-separating set of $D$, that is, for each pair $uv,wz\in A(D)$ $(uv\neq wz)$, $(\{uv\}\cup\omega^-[u])\cap C\neq (\{wz\}\cup\omega^-[w])\cap C$.
  \end{itemize}
\end{definition}

\noindent Hence, a line digraph $LD$ admits a $(1,\leq\ell)$-identifying code if and only if $D$ admits a $(1,\leq\ell)$-arc-identifying code. As a consequence, the minimum size of an identifying code of a digraph $D$, $\overrightarrow{\gamma}^{ID}(LD)$, is equivalent to the minimum size of an arc-identifying code of its line digraph $LD$.

Let $D$ be a digraph. We denote $V^+_{\geq2}(D)=\{v\in V(D): d^+(v)\geq2\}$, and $V^+_{1}(D)=\{v\in V(D): d^+(v)=1\}$. Hence, in particular, if $D$ is a strongly connected digraph, $V(D)=V^+_{1}(D)\cup V^+_{\geq2}(D)$.

\begin{theorem}
 \label{FirstLowerBound}
 Let $D$ be a strongly connected digraph with minimum in-degree $\delta^-\geq2$. Then,
$$\overrightarrow{\gamma}^{ID}(LD)\geq |A(D)|-|V(D)|.$$
\end{theorem}
\begin{proof}
 By Remark \ref{RemarkObsLineDig} $(i)$, $LD$ admits an identifying code. Let $C$ be an arc-identifying code of $D$. Then, by Lemma \ref{RepEnC},
 \begin{equation*}
\begin{split}
  |C|&\geq\sum_{V^+_{\geq2}(D)}(d^+_D(u)-1)\\
  &=\sum_{V^+_{\geq2}(D)}d^+_D(u)-|V^+_{\geq2}(D)|+ \sum_{V^+_{1}(D)}d^+_D(u)-\sum_{V^+_{1}(D)}d^+_D(u)\\
  &= \sum_{V(D)}d^+_D(u)
  -|V^+_{\geq2}(D)|-|V^+_{1}(D)|\\
  &=|A(D)|-|V(D)|.
  \end{split}
\end{equation*}
\end{proof}

\begin{theorem}\label{CharacConditionsOverC}
Let $D$ be a strongly connected digraph of order at least 3, and let $C\subseteq A(D)$. Then, $C$ is an arc-identifying code of $D$ if and only if $C$ satisfies the following conditions:
\begin{itemize}
 \item[(i)] for all $v\in V(D)$, $|\omega^+(v)\setminus C|\leq1$, and if $|\omega^+(v)\setminus C|=1$, then $\omega^-(v)\cap C\neq\emptyset$;
\item[(ii)] for all  $uv\in C$, if $vu\in C$ or $|\omega^+(v)\setminus C|=1$, then $( (\omega^-(v)\cup\omega^-(u))\setminus\{uv, vu\})\cap C\neq\emptyset$.
\end{itemize}
\end{theorem}
\begin{proof}
First suppose that $C$ is an arc-identifying code of $D$.
The first part of $(i)$ follows directly from Lemma \ref{RepEnC}. For the second one, let $v\in V(D)$ be such that $|\omega^+(v)\setminus C|=1$ and let $vx\in\omega^+(v)\setminus C$. Hence, $(\{vx\}\cup\omega^-(v))\cap C=\omega^-(v)\cap C$. Since $C$ is an arc-identifying code, $(\{vx\}\cup\omega^-(v))\cap C\neq\emptyset$, hence $C$ satisfies $(i)$.
To prove $(ii)$, let $uv\in C$ be such that $((\omega^-(u)\cup \omega^-(v))\setminus \{vu,uv\})\cap C=\emptyset$. If $vu\in C$, then $(\{uv\}\cup\omega^-(u))\cap C=\{uv,vu\}=(\{vu\}\cup\omega^-(v))\cap C$, contradicting that $C$ is an arc-identifying code. Hence, $vu\notin C$. If $|\omega^+(v)\setminus C|=1$, let say $\omega^+(v)\setminus C=\{vx\}$, then $(\{uv\}\cup\omega^-(u))\cap C=\{uv\}=N^-[vx]\cap C$, a contradiction. Therefore, $C$ satisfies $(ii)$.

Now, suppose that $C$ is a set of arcs of $D$ satisfying $(i)$ and $(ii)$, and let us show that $C$ is an arc-identifying code.
Let us show that $C$ is an arc-dominating set of $D$. Let $ab\in A(D)$. By $(i)$, $\omega^+(a)\subseteq C$ or $\omega^-(a)\cap C\neq\emptyset$, implying that $(\{ab\}\cup\omega^-(a))\cap C\neq\emptyset$. Therefore, $C$ is an arc-dominating set of $D$.
Next, let us prove that $C$ is an arc-separating set of $D$. On the contrary, suppose that there are two different arcs $ab$ and $cd$, such that $(\{ab\}\cup\omega^-(a))\cap C =(\{cd\}\cup\omega^-(c))\cap C$.
First, let us assume that $ab,cd\notin C$. If we take an arc $uv\in (\{ab\}\cup\omega^-(a))\cap C=(\{cd\}\cup\omega^-(c))\cap C$, then $v=a=c$, implying that $ab,cd\in\omega^+(v)\setminus C$, contradicting $(i)$.
Second, let us assume that $ab\in C$, hence, $c=b$. If $bd\notin C$, by $(ii)$, $((\omega^-(a)\cup \omega^-(b))\setminus \{ba,ab\})\cap C\ne\emptyset$, implying that $(\{ab\}\cup\omega^-(a))\cap C \neq (\{bd\}\cup\omega^-(b))\cap C$, a contradiction.  Therefore, $bd\in C$ implying that  $d=a$   because our assumption $(\{ab\}\cup\omega^-(a))\cap C =(\{cd\}\cup\omega^-(c))\cap C$. Hence, again by $(ii)$, $( (\omega^-(a)\cup\omega^-(b))\setminus\{ab, cd\})\cap C\neq\emptyset$, implying that $(\{ab\}\cup\omega^-(a))\cap C =(\{cd\}\cup\omega^-(c))\cap C$, a contradiction. Therefore, $C$ is an arc-separating set. This completes the proof.
\end{proof}

Now we present an algorithm for constructing an arc-identifying code of a given strongly connected oriented graph with minimum in-degree $\delta^-\geq2$.

\begin{algorithm}\label{Algo}
Constructing an arc-identifying code $C$ of a given strongly connected digraph $D$ with minimum in-degree $\delta^-\geq2$ and without digons.\\
1: Let $U^-:=\{v\in V(D): N^-(v)\subseteq V_1^+(D)\}$, $U:=\emptyset$ and $C:=\emptyset$\\
2: \textbf{while} $U^-\setminus U\neq\emptyset$ \textbf{do}\\
3: let $v\in U^-\setminus U$ and $f\in N^-(v)$\\
4: replace $U$ by $U\cup\{v\}$ and $C$ by $C\cup\{fv\}$\\
5: \textbf{end while}\\
6: let $X:=V_1^+(D)$ and $Y:= U^-$\\
7: let $xy\in A(D)$ such that $x\in V(D)\setminus X$ and $y\in V(D)\setminus Y$\\
8: replace $Y$ by $Y\cup (N^+(x)\setminus\{y\})$, $X$ by $X\cup\{x\}$ and $C$ by $C\cup(\omega^+(x)\setminus\{xy\})$\\
9: \textbf{while} $Y\neq V(D)$ \textbf{do}\\
10: \textbf{while} $N^-(y)\setminus X\neq\emptyset$ \textbf{do}\\
11: let  $t\in N^-(y)\setminus X$ and let $z\in N^+(t)\setminus \{y\}$\\
12: replace $Y$ by $Y\cup(N^+(t)\setminus\{z\})$, $X$ by $X\cup\{t\}$, $C$ by $C\cup(\omega^+(t)\setminus\{tz\})$, $t$ by $x$ and $z$ by $y$\\
13: \textbf{end while}\\
14: \textbf{if} $N^-(y)\setminus X=\emptyset$ \textbf{then}\\
15: choose an arc $uv$ of $D$ such that $v\notin Y$\\
16: replace $Y$ by $Y\cup(N^+(u)\setminus\{v\})$, $X$ by $X\cup\{u\}$, $C$ by $C\cup(\omega^+(u)\setminus\{uv\})$, $u$ by $x$ and $v$ by $y$\\
17: return to 3\\
18: \textbf{end if}\\
19:  \textbf{end while}\\
20: \textbf{if} $Y= V(D)$ \textbf{then}\\
21: \textbf{while} $X\neq V(D)$ \textbf{do}\\
22: let $u\in V(D)\setminus X$ and let $v\in N^+(u)$\\
23: replace $C$ by $C\cup(\omega^+(u)\setminus\{uv\})$, $X$ by $X\cup\{u\}$\\
24: \textbf{end while}\\
25: \textbf{end if}\\
26: return $C$
 \end{algorithm}

\begin{theorem}\label{uper} Let $D$ be an oriented and strongly connected graph with minimum in-degree $\delta^-\ge 2$. Then, Algorithm \ref{Algo} produces a subset $C\subset A(D)$ of $$|C|=|A(D)|-|V(D)|+|\{v\in V(D): N^-(v)\subseteq V^+_1(D)\}|,$$ satisfying the requirements of Theorem \ref{CharacConditionsOverC}.
\end{theorem}
\begin{proof} By construction, given Algorithm \ref{Algo}, we can check that for every $x\in V(D)$ we get $|\omega^+(x)\setminus C|\leq1$. Let $v\in V(D)$. Since the algorithm finishes when the set $Y$ is equal to $V(D)$, it follows that $v\in Y$ at a certain step of the algorithm. Then, $v\in N^+(t)\setminus \{z\}$ for a certain $t$ and $z$ in the algorithm, which implies that $tv\in \omega^+(t)\setminus \{tz\}\subset C$. Then,  $\omega^-(v)\cap C\ne \emptyset$ and Theorem \ref{CharacConditionsOverC} $(i)$ holds.
Finally, since $D$ is oriented, for all $uv\in C$, clearly $vu\not\in A(D)$, and we have $|(\omega^-(u)\cup (\omega^-(v)\setminus \{uv\}))\cap C|\ge 1$ because $\omega^-(u)\cap C\neq \emptyset$. Hence, Theorem \ref{CharacConditionsOverC} $(ii)$ also holds. Therefore, $C$ is an arc-identifying code of $D$ and $|C|=|A(D)|-|V(D)|+|\{v\in V(D): N^-(v)\subseteq V^+_1(D)\}|$.
\end{proof}

As a consequence of Theorems \ref{FirstLowerBound} and \ref{uper}, we can conclude the following corollary.

\begin{corollary} Let $D$ be a strongly connected oriented graph with minimum in-degree $\delta^-\ge 2$. Then, the following assertions hold.
\begin{itemize}
  \item[(i)] $\overrightarrow{\gamma}^{ID} (LD)=|A(D)|-|V(D)|+|\{v\in V(D): N^-(v)\subseteq V^+_1(D)\}|$ if $\delta^+=1$;
  \item[(ii)] $\overrightarrow{\gamma}^{ID} (LD)=|A(D)|-|V(D)|$ if $\delta^+\geq2$.
\end{itemize}
\end{corollary}

Next, we also present a result for all Hamiltonian digraphs of minimum degree at least two, not necessarily oriented.

\begin{theorem}
\label{Hamil}
 Let $D$ be a Hamiltonian strongly connected digraph with minimum in-degree $\delta^-\ge 3$ and out-degree $\delta^+\ge 2$. Then, $\gamma^{ID} (LD)=|A(D)|-|V(D)|$.
\end{theorem}
\begin{proof} Let $H=(x_1,x_2, \ldots, x_n,x_1) $ denote a Hamiltonian cycle of $D$. Let $C=A(D)\setminus A(H)$. Let us show that $C$ satisfies the requirements of Theorem \ref{CharacConditionsOverC}. By definition of $C$, $|\omega^+(v)\setminus C|=1$ and $|\omega^-(v)\setminus C|=1$. Since $\delta\geq2$, Theorem \ref{CharacConditionsOverC} $(i)$ follows directly. To show Theorem \ref{CharacConditionsOverC} $(ii)$,
observe that $|\omega^-(v)\setminus C|=1$ implies that $|\omega^-(v)\cap C|=d^-(v)-1\geq2$ because $\delta^-\geq3$. Therefore, for all $uv\in C$, $( (\omega^-(v)\cup\omega^-(u))\setminus\{uv, vu\})\cap C\neq\emptyset$ Theorem \ref{CharacConditionsOverC} $(ii)$ holds. Thus, $C$ is an arc-identifying code of $D$ and $\gamma^{ID} (LD)=|A(D)|-|V(D)|$ by Theorem \ref{FirstLowerBound}.
\end{proof}

\begin{corollary}
\label{CoroKautz-d3}
The identifying number of a Kautz digraph $K(d,k)$ is $\gamma^{ID} (K(d,k))=d^{k}-d^{k-2}$ for $d\geq3$ and $k\geq2$.
\end{corollary}
\begin{proof} Note that $K(d,2)=LK_{d+1}$. Since $K_{d+1}$ is Hamiltonian and $d\ge 3$, by Theorem \ref{Hamil}, $\gamma^{ID} (K(d,2))=\gamma^{ID} (LK_{d+1})=d(d+1)-(d+1)=d^2-1$, and the result holds for $k=2$.
For any $k\geq3$, the Kautz digraph $K(d,k)=L^{k-1}K_{d+1}=LL^{k-2}K_{d+1}=LK(d,k-1)$. Since $K(d,k-1)$ is a Hamiltonian digraph and $d\geq 3$, by Theorem \ref{Hamil}, $\gamma^{ID} (K(d,k))=\gamma^{ID} (LK(d,k-1))=d^{k}+d^{k-1}-(d^{k-1}+d^{k-2})=d^{k}-d^{k-2}$ and the result holds.
\end{proof}

To extend the Corollary \ref{CoroKautz-d3} to $K(2,k)$ we need the 1-factorization of Kautz digraphs obtained by Tvrd\'ik \cite{T94}. This 1-factorization uses the following operation.

\begin{definition}\cite{T94}
\label{sigma}
If $x=x_1\dots x_k\in V(K(d,k))$, then
\begin{itemize}
\item $\sigma_1(x)=x_2\dots x_{k-1}x_kx_1$ if $x_1\neq x_k$
\item $\sigma_1(x)=x_2\dots x_{k-1}x_kx_2$ if $x_1= x_k$
\end{itemize}

Let $Inc:V(K(d,k))\times \mathbb{Z}_d\rightarrow V(K(d,k))$ denote a binary operation such that
$$Inc(x_1\dots x_{k-1}x_k,i)=x_1\dots x_{k-1}x'_k$$ where
$$x'_k=\left\{\begin{array}{ll}
        x_k+i~ \mod(d+1) & \mbox{if } x_{k-1}>x_k\mbox{ and } x_{k-1}>x_k+i\\
        &\mbox{or }x_{k-1}<x_k \mbox{ and }x_{k-1}+d+1>x_k+i;\\
        x_k+i+1~\mod(d+1) &\mbox{ otherwise.}
      \end{array}\right.  $$
 Then, the generalized $K$-shift operation is defined as follows:
 $$\begin{array}{ll} \sigma^{+i}_1(x)&=Inc(\sigma_1(x),i),\\
\sigma^{+i}_k &= \sigma^{+i}_1 \circ\sigma^{+i}_{k-1}.\end{array}  $$
\end{definition}
\begin{theorem}\label{factor}\cite{T94} The arc set of $K(d,k)$ can be partitioned into $d$ 1-factors $\mathcal{F}_0,\ldots \mathcal{F}_{d-1}$ such that  the cycles of $\mathcal{F}_i$ are closed under the operation $\sigma^{+i}_1$.
\end{theorem}
\begin{theorem}
\label{TheoKautz-d2}
The identifying number of a Kautz digraph $K(2,k)$ is $\gamma^{ID} (K(2,k))=2^{k}-2^{k-2}$ for $k\geq2$.
\end{theorem}
\begin{proof} It is easy to check in Figure \ref{fig:exemple} that $C=\{uv, vw, wu\}$ is an identifying code of $ K(2,2)$, then $\gamma^{ID} (K(2,2))=3$ and the theorem holds for $k=2$. Suppose that $k \ge     3$ and let us study the Kautz digraph $K(2,k-1)$.
By Theorem \ref{factor}, we can consider a partition of the arcs of $K(2,k-1)$ into two 1-factors $\mathcal{F}_0$ and $\mathcal{F}_1$, such that the cycles of $\mathcal{F}_i$ are closed under the operation called $\sigma_1^{+i}$, given in Definition \ref{sigma}. It is not difficult to see that  the relation $\sigma_1^{+0}$ preserves digons, implying that all the digons of $K(2,k-1)$ belong to the family $\mathcal{F}_0$. Hence, since $\mathcal{F}_1$ is a 1-factor of $K(2,k-1)$, it is clear that the set of arcs in $\mathcal{F}_1$, say $A_1$, satisfies the conditions of Theorem \ref{CharacConditionsOverC}. Therefore,  $A_1$ is an arc-identifying code of $K(2,k-1)$, that is, an identifying code of $K(2,k)$. Therefore, $\gamma^{ID} (K(2,k))=|A_1|=|V(K(2,k-1)|=3\cdot 2^{k-2}=2^{k}-2^{k-2}$, and the proof is complete.
\end{proof}

\bibliographystyle{plain}

\end{document}